\documentclass[11pt, a4paper, oneside]{amsart}

\usepackage[english]{babel}
\usepackage{amsmath, amsthm, amsfonts, mathrsfs, amssymb}
\usepackage{mathtools}
\mathtoolsset{centercolon}
\usepackage{booktabs}
\usepackage[shortlabels]{enumitem}
\usepackage[colorlinks, citecolor = blue]{hyperref}

\newcommand{\N}{\ensuremath{\mathbb{N}}}

\newcommand{\R}{\ensuremath{\mathbb{R}}}
\newcommand{\C}{\ensuremath{\mathbb{C}}}

\newcommand{\mb}{\mathbf}
\newcommand{\mc}{\mathcal}

\DeclarePairedDelimiter\abs{\lvert}{\rvert}
\DeclarePairedDelimiter\brac[]
\DeclarePairedDelimiter\cbrace\{\}
\DeclarePairedDelimiter\ha()
\DeclarePairedDelimiter{\ip}\langle\rangle
\DeclarePairedDelimiter{\nrm}\lVert\rVert

\newcommand{\nrmb}[1]{\bigl\|#1\bigr\|}

\newcommand{\hab}[1]{\bigl(#1\bigr)}
\newcommand{\cbraceb}[1]{\bigl\{#1\bigr\}}
\newcommand{\ipb}[1]{\bigl\langle#1\bigr\rangle}
\newcommand{\bracb}[1]{\bigl[#1\bigr]}

\newcommand{\nrms}[1]{\Bigl\|#1\Bigr\|}

\newcommand{\has}[1]{\Bigl(#1\Bigr)}
\newcommand{\cbraces}[1]{\Bigl\{#1\Bigr\}}

\DeclareMathOperator{\sgn}{sgn}

\DeclareMathOperator{\ind}{\mathbf{1}}
\DeclareMathOperator{\UMD}{UMD}

\newcommand{\dd}{\hspace{2pt}\mathrm{d}}

\newcommand{\inc}{\mb{\phi}}   

\newtheorem{theorem}{Theorem}
\newtheorem{corollary}[theorem]{Corollary}
\newtheorem{lemma}[theorem]{Lemma}
\newtheorem{proposition}[theorem]{Proposition}

\theoremstyle{remark}
\newtheorem{remark}[theorem]{Remark}

\theoremstyle{definition}

\numberwithin{theorem}{section}
\numberwithin{equation}{section}

\allowdisplaybreaks

\title[The \texorpdfstring{$\ell^\MakeLowercase{s}$}{\MakeLowercase{ls}}-boundedness of a family of integral operators]{\texorpdfstring{The $\ell^\MakeLowercase{s}$}{\MakeLowercase{ls}}-boundedness of a family of integral operators on UMD Banach function spaces}
\author{Emiel Lorist}
\thanks{The author is supported by the VIDI subsidy 639.032.427 of the Netherlands Organisation for Scientific Research (NWO)}

\address{Delft Institute of Applied Mathematics \\ Delft University of Technology \\ P.O. Box 5031\\ 2600 GA Delft \\The Netherlands}
\email{e.lorist@tudelft.nl}

\begin{document}

\begin{abstract}
 We prove the $\ell^s$-boundedness of a family of integral operators with an operator-valued kernel on $\UMD$ Banach function spaces. This generalizes and simplifies the earlier work by Gallarati, Veraar and the author \cite{GLV16}, where the $\ell^s$-boundedness of this family of integral operators was shown on Lebesgue spaces. The proof is based on a characterization of $\ell^s$-boundedness as weighted boundedness by Rubio de Francia.
\end{abstract}




\dedicatory{Dedicated to Ben de Pagter on the occasion of his 65th birthday.}

\maketitle

\section{Introduction}
Over the past decades there has been a lot of interest in the $L^p$-maximal regularity of PDEs. Maximal $L^p$-regularity  of the abstract Cauchy problem
\begin{align}\label{eq:MR}
\left\{
    \begin{aligned}
      u'(t) +Au(t)&= f(t), \qquad t \in (0,T]\\
    u(0) &= x,
    \end{aligned}
  \right.
\end{align}
where $A$ is a closed operator on a Banach space $X$, means that for all $f \in L^p((0,T];X)$ the solution $u$ has ``maximal regularity'', i.e. both $u'$ and $Au$ are in $L^p((0,T];X)$. Maximal $L^p$-regularity can for example be used to solve quasi-linear and fully nonlinear PDEs by linearization techniques combined with the contraction mapping principle, see e.g. \cite{AM05,CP01,Lu95,Pr02}.

In the breakthrough work of Weis \cite{We01,We01b}, an operator theoretic characterization of maximal $L^p$-regularity on $\UMD$ Banach spaces was found in terms of the $\mc{R}$-boundedness of the resolvents of $A$ on a sector.
$\mc{R}$-boundedness is a random boundedness condition on a family of operators which is a strengthening of uniform boundedness. We refer to \cite{CDSW00,HNVW17} for more information on $\mc{R}$-boundedness.

In \cite{GV17b, GV17} Gallarati and Veraar developed a new approach to maximal $L^p$-regularity for the case where the operator $A$ in \eqref{eq:MR} is time-dependent and $t \mapsto A(t)$ is merely assumed to be measurable. In this new approach $\mc{R}$-boundedness is once again one of the main tools. For their approach the $\mc{R}$-boundedness of the family of integral operators $\cbrace{I_k:k \in \mc{K}}$ on $L^p(\R;X)$ is required. Here $I_k$ is defined for $f \in L^p(\R;X)$ as
\begin{equation*}
  I_kf(t) := \int_{-\infty}^t k(t-s)T(t,r)f(r)\dd r ,\qquad t \in \R,
\end{equation*}
where $T(t,s)$ is the two-parameter evolution family associated to $A(t)$ and $\mc{K}$ contains all kernels $k \in L^1(\R)$ such that $\abs{k} * \abs{g} \leq Mg$ for all simple $g:\R \to \C$.

In the literature there are many $\mc{R}$-boundedness results for integral operators, see \cite[Chapter 8]{HNVW17} for an overview. However none of these are applicable to the operator family of $\cbrace{I_k:k \in \mc{K}}$. Therefore in \cite{GLV16} Gallarati, Veraar and the author show a sufficient condition for the $\mc{R}$-boundedness of $\cbrace{I_k:k \in \mc{K}}$ on $L^p(\R;X)$ in the special case where $X = L^q$. This is done through the notion of $\ell^s$-boundedness, which states that for all finite sequences $(I_{k_j})_{j=1}^n$ in  $\cbrace{I_k:k \in \mc{K}}$ and $(x_j)_{j=1}^n$ in $X$ we have
\begin{equation*}
  \nrms{\has{\sum_{j=1}^n  \abs{I_{k_j} x_j}^s}^{1/s}}_{X} \lesssim \nrms{\has{\sum_{j=1}^n  \abs{x_j}^s}^{1/s}}_{X}.
\end{equation*}
For $s=2$ this notion coincides with $\mc{R}$-boundedness as a consequence of the Kahane-Khintchine inequalities.

Our main contribution is the generalization of the main result in \cite{GLV16} to the setting of $\UMD$ Banach function spaces $X$. For the proof we will follow the general scheme of \cite{GLV16} with some simplifications. As in case $X=L^q$, for any $\UMD$ Banach function space the notions of $\ell^2$-boundedness and $\mc{R}$-boundedness coincide, so the following theorem in particular implies the $\mc{R}$-boundedness of $\cbrace{I_k:k \in \mc{K}}$.

 \begin{theorem}
\label{theorem:mainintro}
Let $X$ be a $\UMD$ Banach function space and $p \in (1,\infty)$. Let $T\colon\R\times \R \to \mc{L}(X)$ be such that the family of operators
$$\cbraceb{T(t,r):t,r\in \R}$$
is $\ell^s$-bounded for all $s \in (1, \infty)$.
Then $\cbrace{I_k:k \in \mc{K}}$ is $\ell^s$-bounded on $L^p(\R;X)$ for all $s \in (1, \infty)$.
\end{theorem}

We will prove Theorem \ref{theorem:mainintro} in a more general setting in Section \ref{section:integral}. In particular we allow weights in time, which in applications for example allow  rather rough initial values (see e.g. \cite{KPW10, Li17b, MS12, PS04}).

For certain $\UMD$ Banach function spaces the $\ell^s$-bounded\-ness assumption in Theorem \ref{theorem:mainintro} can be checked by weighted extrapolation techniques, see Corollary \ref{corollary:main} and Remark \ref{remark:maincorollary}.

\subsection*{Notation}
For a measure space $(S,\mu)$ we denote the space of all measurable functions by $L^0(S)$. We denote the Lebesgue measure of a Borel set $E \in \mc{B}(\R^d)$ by $\abs{E}$. For Banach spaces $X$ and $Y$ we denote the vector space of bounded linear operators from $X$ to $Y$ by $\mc{L}(X,Y)$ and we set $\mc{L}(X):=\mc{L}(X,X)$. For a operator family $\Gamma\subset \mc{L}(X,Y)$ we set $\Gamma^*:= \cbrace{T^*:T\in \Gamma}$. For $p \in [1,\infty]$ we let $p'\in [1,\infty]$ be such that $\frac{1}{p}+\frac{1}{p'}=1$.

Throughout the paper we write $C_{a,b,\cdots}$ and $\phi_{a,b,\cdots}$ to denote a constant and a  nondecreasing function on $[1,\infty)$ respectively, which only depend on the parameters $a,b,\cdots$ and the dimension $d$ and which may change from line to line.

\subsection*{Acknowledgement} The author would like to thank Mark Veraar for carefully reading the draft version of this paper.

\section{Preliminaries}\label{section:preliminaries}
\subsection{Banach function spaces}
Let $(S,\mu)$ be a $\sigma$-finite measure space. An order ideal $X$ of $L^0(S)$ equipped with a norm $\nrm{\cdot}_X$ is called a {\emph{Banach function space}} if it has the following properties:
\begin{enumerate}[(i)]
  \item \textit{Compatibility:} If $\xi,\eta \in L^0(S)$ with $\abs{\xi} \leq \abs{\eta}$, then $\nrm{\xi}_X \leq \nrm{\eta}_X$
  \item \textit{Weak order unit:} There is a $\xi \in X$ with $\xi>0$.
  \item \textit{Fatou property:} If $0 \leq \xi_n \uparrow \xi$ for $(\xi_n)_{n=1}^\infty$ in $X$, $\xi \in L^0(S)$ and $\sup_{n \in \N} \nrm{\xi_n}_X<\infty$, then $\xi \in X$ and $\nrm{\xi}_X = \sup_{n\in \N} \nrm{\xi_n}_X$.
\end{enumerate}
A Banach function space is called \emph{order continuous} if for any sequence $0\leq \xi_n \uparrow \xi \in X$ we have $\nrm{\xi_n-\xi}_X \rightarrow 0$. Every reflexive Banach function space is order continuous. Order continuity ensures that the dual of $X$ is also a a Banach function space. For a thorough introduction to Banach function spaces we refer to \cite[section 1.b]{LT79} or \cite[Chapter 1]{BS88}.

A Banach function space $X$ is said to be \emph{$p$-convex} for $p \in [1,\infty]$ if
\begin{equation*}
  \nrms{\has{\sum_{j=1}^n \abs{\xi_k}^p}^{1/p}}_X \leq \has{\sum_{j=1}^n \nrm{\xi_j}_X^p}^{1/p}
\end{equation*}
for all $\xi_1,\cdots,\xi_n \in X$ with the sums replaced by suprema if $p = \infty$.
The defining inequality for $p$-convexity often includes a constant, but $X$ can always be renormed such that this constant  equals $1$. If a Banach function space is $p$-convex for some $p\in [1,\infty]$, then $X$ is also $q$-convex for all $q \in [1,q]$.

For a $p$-convex Banach function space $X$ we can define another Banach function space by
\begin{equation*}
  X^p := \cbraceb{\abs{\xi}^p \sgn \xi: \xi \in X} = \cbraceb{\xi \in L^0(S):\abs{\xi}^{1/p} \in X}
\end{equation*}
equipped with the norm $\nrm{\xi}_{X^p} := \nrmb{\abs{\xi}^{1/p}}_X^p$. We refer the interested reader to \cite[section 1.d]{LT79} for an introduction to $p$-convexity.

\subsection{\texorpdfstring{$\ell^s$}{ls}-boundedness}
Let $X$ and $Y$ be Banach functions spaces and let $\Gamma \subseteq \mc{L}(X,Y)$ be a family of operators.
We say that $\Gamma$ is \emph{$\ell^s$-bounded} if for all finite sequences $(T_j)_{j=1}^n$ in  $\Gamma$ and $(x_j)_{j=1}^n$ in $X$ we have
\begin{equation*}
  \nrms{\has{\sum_{j=1}^n  \abs{T_j x_j}^s}^{1/s}}_{Y} \leq C \, \nrms{\has{\sum_{j=1}^n  \abs{x_j}^s}^{1/s}}_{X}.
\end{equation*}
with the sums replaced by suprema if $s =\infty$. The least admissible constant $C$ will be denoted by $\brac{\Gamma}_{\ell^s}$.

Implicitly $\ell^s$-boundedness is a classical tool in harmonic analysis for operators on $L^p$-spaces (see e.g. \cite[Chapter V]{GR85} and \cite{Gr08,Gr09}). For Banach function spaces the notion was introduced in \cite{We01} under the name $\mc{R}_s$-boundedness, underlining its connection to the more well-known notion of $\mc{R}$-boundedness. An extensive study of
$\ell^s$-boundedness can be found in \cite{KU14} and for a comparison between $\ell^2$-boundedness and $\mc{R}$-boundedness we refer to \cite{KVW16}.

\begin{lemma}\label{lemma:ls-boundedness}
  Let $X$ and $Y$ be Banach function spaces and let $\Gamma \subseteq \mc{L}(X,Y)$.
  \begin{enumerate}[(i)]
 \item\label{it:interpolatels} Let $1 \leq s_0<s_1 \leq \infty$ and assume that $X$ and $Y$ are order continuous. If $\Gamma$ is $\ell^{s_0}$- and $\ell^{s_1}$-bounded, then $\Gamma$ is $\ell^s$-bounded for all $s \in [s_0,s_1]$ with $\brac{\Gamma}_{\ell^s} \leq \max\cbraceb{\brac{\Gamma}_{\ell^{s_0}},\brac{\Gamma}_{\ell^{s_1}}}$
  \item \label{it:dualizels} Let $s \in [1,\infty]$ and assume that $\Gamma$ is $\ell^s$-bounded. Then the adjoint family $\Gamma^*$ is $\ell^{s'}$-bounded with
      $\brac{\Gamma^*}_{\ell^{s'}}=\brac{\Gamma}_{\ell^s}$
  \end{enumerate}
\end{lemma}
\begin{proof}
Lemma \ref{lemma:ls-boundedness}\ref{it:interpolatels} follows from Calder\'on's theory of complex interpolation of vector-valued function spaces, see  \cite{Ca64} or \cite[Proposition 2.14]{KU14}. Lemma \ref{lemma:ls-boundedness}\ref{it:dualizels} is direct from the identification $X(\ell^s_n)^* = X^*(\ell^{s'}_n)$, see \cite[Section 1.d]{LT79} or \cite[Proposition 2.17]{KU14}
\end{proof}

 The following characterization of $\ell^s$-boundedness for $s \in [1,\infty)$ will be one of the key ingredients of our main result. This characterization relating $\ell^s$-bounded\-ness to a certain weighted bounded\-ness comes from the work of Rubio de Francia \cite{GR85, Ru84,Ru86}.
\begin{proposition}\label{proposition:Rubiols}
Let  $s \in [1,\infty)$ and let $X$ and $Y$ be $s$-convex order continuous Banach function spaces over  $(S_X,\mu_X)$ and $(S_Y,\mu_Y)$ respectively. Let $\Gamma\subseteq \mc{L}(X)$ and take $C>0$.
  Then the following are equivalent:
  \begin{enumerate}[(i)]
  \item $\Gamma$ is $\ell^s$-bounded  with $\brac{\Gamma}_{\ell^s} \leq C$.
  \item For all nonnegative $u \in (Y^s)^*$, there exists a nonnegative $v \in (X^s)^*$ with  $\nrm{v}_{(Y^s)^*} \leq \nrm{u}_{(X^s)^*}$ and
    \begin{equation*}
      \has{\int_{S_Y} \abs{T(\xi)}^s u \dd\mu_Y}^{1/s} \leq C \, \has{\int_{S_X} \abs{\xi}^s v \dd\mu_X}^{1/s}
    \end{equation*}
    for all $\xi \in X$ and  $T \in \Gamma$.
  \end{enumerate}
\end{proposition}

\begin{proof}
  The statement is a combination of \cite[Lemma 1, p. 217]{Ru86} and \cite[Theorem  VI.5.3]{GR85}, which for $X=Y$ is proven \cite[Lemma 3.4]{ALV17}. The statement for $X \neq Y$ is can be extracted from the proof of \cite[Lemma 3.4]{ALV17} and can in full detail be found in \cite[Proposition 6.1.3]{Lo16}
\end{proof}

\subsection{Muckenhoupt weights}
  A locally integrable function $w:\R^d\to (0,\infty)$ is called a \emph{weight}. For $p \in (1,\infty)$ and a weight $w$ we let $L^p(w)$ be the space of all $f \in L^0(\R^d)$ such that
\begin{equation*}
  \nrm{f}_{L^p(w)}:= \has{\int_{\R^d}\abs{f}^pw}^{1/p}<\infty.
  \end{equation*}
We will say that a weight $w$ lies in the \emph{Muckenhoupt class $A_p$} and write $w\in A_p$ if it satisfies
\[
[w]_{A_p}:=\sup_{Q}\frac{1}{\abs{Q}}\int_Q w \cdot \has{\frac{1}{\abs{Q}}\int_Q w^{1-p'}}^{p-1}<\infty,
\]
where the supremum is taken over all cubes  $Q\subseteq\R^d$ with sides parallel to the coordinate axes.

\begin{lemma}\label{lemma:weights}
  Let $p \in (1,\infty)$ and $w \in A_p$.
  \begin{enumerate}[(i)]
  \item \label{it:weights1} $w\in A_q$ for all $q \in (p,\infty)$ with $[w]_{A_q}\leq [w]_{A_p}$.\vspace{5pt}
  \item \label{it:weights2} $w^{1-p'} \in A_{p'}$ with $[w]_{A_p}^{1/p}= [w^{1-p'}]^{1/p'}_{A_p'}$.\vspace{5pt}
  \item \label{it:weights3} $w \in A_{p-\varepsilon}$ for $\varepsilon = \frac{1}{\inc_p([w]_{A_p})}$
  with $[w]_{A_{p-\varepsilon}} \leq \inc_{p}\ha{[w]_{A_p}}$.
  \end{enumerate}
\end{lemma}
The first two properties of Lemma \ref{lemma:weights} follow directly from the definition. The third is for example proven in \cite[Exercise 9.2.4]{Gr09}. For a more thorough introduction to Muckenhoupt weights we refer to \cite[Chapter 9]{Gr09}.

\subsection{The UMD property}\label{section:UMD}
A Banach space $X$ is said to have the $\UMD$ property if the martingale difference sequence of any finite martingale in $L^p(\Omega;X)$ is unconditional for some (equivalently all) $p \in (1,\infty)$. We will work with $\UMD$ Banach function spaces, of which standard examples include reflexive Lebesgue, Lorentz and Orlicz spaces. In this Festschrift it is shown that reflexive Musielak-Orlicz spaces, so in particular reflexive variable Lebesgue spaces, have the $\UMD$ property, see \cite{LVY18}. The $\UMD$ property implies reflexivity, so in particular $L^1$ and $L^\infty$ do not have the $\UMD$ property. For a thorough introduction to the theory of $\UMD$ Banach spaces we refer to \cite{Bu01,HNVW16}.

For an order continuous Banach function space $X$ over $(S,\mu)$ there is also a characterization of the $\UMD$ property in terms of the \emph{lattice Hardy--Littlewood maximal operator}, which for simple functions $f\colon\R^d \to X$ is given by
\begin{equation*}
  \widetilde{M}f(x) := \sup_{Q \ni x} \frac{1}{\abs{Q}}\int_Q\abs{f(y)}\dd y, \qquad x\in \R^d
\end{equation*}
where the supremum is taken pointwise in $S$ and over all cubes $Q \subseteq \R^d$ with sides parallel to the coordinate axes (see \cite{GMT93} or \cite[Lemma 5.1]{HL18b} for a detailed definition of $\widetilde{M}$). It is a deep result by Bourgain \cite{Bo84} and Rubio de Francia \cite{Ru86} that $X$ has the $\UMD$ property if and only if $\widetilde{M}$ is bounded on $L^p(\R^d;X)$ and $L^p(\R^d;X^*)$ for some (equivalently all) $p \in (1,\infty)$. For weighted $L^p$-spaces we have the following proposition, which was proven in \cite{GMT93}. The increasing dependence on $[w]_{A_p}$ is shown in \cite[Corollary 5.3]{HL18b}.
\begin{proposition}\label{proposition:HL}
  Let $X$ be a $\UMD$ Banach function space, $p \in (1,\infty)$ and $w \in A_p$. Then for all $f \in L^p(w;X)$ we have
  \begin{equation*}
    \nrmb{\widetilde{M}f}_{L^p(w;X)} \leq \inc_{X,p}\hab{[w]_{A_p}}\nrm{f}_{L^p(w;X)}.
  \end{equation*}
\end{proposition}

  The $\UMD$ property of a Banach function space $X$ also implies that $X^q$ has the $\UMD$ property for a $q>1$, which is a deep result by Rubio de Francia \cite[Theorem 4]{Ru86}.
  \begin{proposition} \label{proposition:rubioUMD}Let $X$ be a $\UMD$ Banach function space. Then there is a $p>1$ such that $X$ is $p$-convex and $X^q$ is a $\UMD$ Banach function space for all $q \in [1,p]$.
  \end{proposition}

\section{Integral operators with an operator-valued kernel}\label{section:integral}
Before turning to our main result on the $\ell^s$-boundedness of a family of integral operators on $L^p(w;X)$ with operator-valued kernels, we will first study the $\ell^s$-boundedness of a family of convolution operators on $L^p(w;X)$  with scalar-valued kernels. For this define
\begin{equation*}
  \mc{K} :=   \cbrace{k\in L^1(\R^d): \abs{k} * \abs{f}\leq Mf \text{ a.e. for all simple} \, f :\R^d \to \C}.
\end{equation*}
  As an example any radially decreasing $k \in L^1(\R^d)$ with $\nrm{k}_{L^1(\R^d)}\leq 1$ is an element of $\mc{K}$. For more examples see \cite[Chapter 2]{Gr08} and \cite[Proposition 4.6]{NVW15}.

Let $X$ be a Banach function space. For a kernel $k \in \mc{K}$ and a simple function $f\colon\R^d \to X$ we define
\begin{equation*}
  T_kf := k*f = \int_{\R^d}k(x-y)f(y)\dd y.
\end{equation*}
As
$$\nrm{T_k f}_X \leq \abs{k}*\nrm{f}_X \leq M \hab{\nrm{f}_X},$$ and since the Hardy-Littlewood maximal operator $M$ is bounded on $L^p(w)$ for all $p \in (1,\infty)$ and $w \in A_p$, $T_k$ extends to a bounded linear operator on $L^p(w;X)$ by density. This argument also shows that the family of convolution operators given by $\Gamma:= \cbrace{T_k:k \in \mc{K}}$ is uniformly bounded on $L^p(w;X)$.

If $X$ is a $\UMD$ Banach function space we can say more. The following lemma was first developed by van Neerven, Veraar and Weis in \cite{NVW12, NVW15} in connection to stochastic maximal regularity. As in \cite{NVW12, NVW15}, the endpoint case $s=1$ will play a major role in the proof of our main theorem in the next section.
\begin{proposition}
\label{proposition:kernel1sbdd}
Let $X$ be a $\UMD$ Banach function space, $s \in [1,\infty]$, $p\in (1, \infty)$ and $w\in A_p$. Then $\Gamma =\cbrace{T_k:k \in \mc{K}}$  is $\ell^{s}$-bounded on $L^p(w;X)$ with
\begin{equation*}
  [\Gamma]_{\ell^s} \leq \inc_{X,p}\hab{[w]_{A_p}}.
\end{equation*}
\end{proposition}
The proof is a weighted variant of \cite[Theorem 4.7]{NVW15}, which for the special case where $X$ is an iterated Lebesgue space is presented in \cite[Proposition 3.6]{GLV16}. For convenience of the reader we sketch the proof in the general case.

\begin{proof}
 As $X$ is reflexive and therefore order-continuous, $\widetilde{M}$ is well-defined on $L^p(w;X)$ and we have $T_k f \leq \widetilde{M} f$ pointwise a.e. for all simple $f\colon\R^d \to X$.

If $s = \infty$ take simple functions $f_1,\cdots,f_n\in L^p(w;X)$  and $k_1,\cdots,k_n \in \mc{K}$. Using Proposition \ref{proposition:HL} we have
  \begin{align*}
  \nrms{\sup_{1 \leq j \leq n} \abs{T_{k_j}f_j}}_{L^p(w;X)} &\leq \nrms{\sup_{1\leq j \leq n}\widetilde{M}f_j(x)}_{L^p(w;X)}\\
  &\leq \nrms{\widetilde{M}\has{\sup_{1\leq j \leq n}\abs{f_j}}(x)}_{L^p(w;X)}\\
  &\leq \inc_{X,p} \hab{[w]_{A_{p}}} \,\nrms{\sup_{1 \leq j \leq n} \abs{f_j}}_{L^p(w;X)}.
\end{align*}
The result now follows by the density of simple functions in $L^p(w;X)$.

If $s=1$ we use duality. Note that since $X$ is reflexive we have $L^p(w;X)^* = L^{p'}(w';X^*)^*$ with $w' = w^{1-p'}$ under the duality pairing
\begin{equation}\label{eq:duality}
  \ip{f, g}_{L^{p}(w;X), L^{p'}(w';X^*)} = \int_{\R^d}   \ipb{f(x), g(x)}_{X,X^*}  \dd x
\end{equation}
by Lemma \ref{lemma:weights}\ref{it:weights2} and \cite[Corollary 1.3.22]{HNVW16}. One can routinely check that $T_k^* = T_{\tilde{k}}$ with $\tilde{k}(x)=k(-x)$ and that $k \in \mc{K}$ if and only if $\tilde{k} \in \mc{K}$. Since $X^*$ is also a $\UMD$ Banach function space (see \cite[Proposition 4.2.17]{HNVW16}) we know from the case $s=\infty$ that the adjoint family $\Gamma^*$ is $\ell^\infty$-bounded on $L^{p'}(\R^d,w';X^*)$, so the result follows by Lemma \ref{lemma:ls-boundedness}\ref{it:dualizels}.
Finally if $s \in (1,\infty)$ the result follows by Lemma \ref{lemma:ls-boundedness}\ref{it:interpolatels}.
\end{proof}

 With these preparations done we can now introduce the family of integral operators with operator-valued kernel that we will consider.
 Let $X$ and $Y$ be a Banach function space and let $\mc{T}$ be a family of operators $\R^d\times \R^d \to \mc{L}(X,Y)$ such that $(x,y)\mapsto T(x,y)\xi$ is measurable for all $T \in \mc{T}$ and $\xi\in X$. The integral operators that we will consider are for simple $f:\R^d \to X$ given by
 \begin{equation*}
I_{k,T} f(x)=\int_{\R^d} k(x-y)T(x,y)f(y)\dd y
\end{equation*}
with $k \in \mc{K}$ and $T \in \mc{T}$. If $\nrm{T(x,y)}_{\mc{L}(X,Y)} \leq C$ for all $T\in \mc{T}$ and $x,y \in \R^d$, we have
$$\nrm{I_{k,T} f}_X \leq C \, \abs{k}*\nrm{f}_X \leq C \,M \hab{\nrm{f}_X}.$$ So as before $I_{k,T}$ extends to a bounded linear operator from $L^p(w;X)$ to $L^p(w;Y)$ for all $p \in (1,\infty)$ and $w\in A_p$, and
$$\mc{I}_{\mc{T}}:= \cbraceb{I_{k,T}:k \in \mc{K}, T \in \mc{T}}$$
is uniformly bounded. For the details see \cite[Lemma 3.9]{GLV16}.

 If $X$ and $Y$ are Hilbert spaces, this implies that $\mc{I}_{\mc{T}}$ is also $\ell^2$-bounded from $L^2(\R^d;X)$ to $L^2(\R^d;Y)$, as these notions coincide on Hilbert spaces. However if $X$ and $Y$ are not Hilbert spaces, but a $\UMD$ Banach function space or if we move to weighted $L^p$-spaces, the $\ell^2$-boundedness of $\mc{I}_{\mc{T}}$ is a lot more delicate.

Our main theorem is a quantitative and more general version of Theorem \ref{theorem:mainintro} in the introduction:

\begin{theorem}
\label{theorem:main}
Let $X$ and $Y$ be a $\UMD$ Banach function spaces and let $p,s \in (1,\infty)$. Let $\mc{T}$ be a family of operators $\R^d\times \R^d \to \mc{L}(X,Y)$ such that
\begin{enumerate}[(i)]
\item $(x,y)\mapsto T(x,y)\xi$ is measurable for all $T \in \mc{T}$ and $\xi\in X$.
\item The family of operators $\widetilde{\mc{T}} := \cbrace{T(x,y):T\in \mc{T},\, x,y\in \R^d}$ is $\ell^\sigma$-bounded for all $\sigma\in (1, \infty)$.
\end{enumerate} Then $\mc{I}_{\mc{T}}$ is $\ell^s$-bounded from $L^p(w;X)$ to $L^p(w;Y)$ for all $w \in A_p$  with
\begin{align*}
  \brac{\mc{I}_{\mc{T}}}_{\ell^s} &\leq\inc_{X,Y,p}\hab{[w]_{A_p}} \, \max\cbraceb{ \bracb{\widetilde{\mc{T}}}_{\ell^\sigma},\bracb{\widetilde{\mc{T}}}_{\ell^{\sigma'}} }, \qquad &&\sigma =  1+\frac{1}{\inc_{p,s}\,[w]_{A_p}}\\
  &\leq\inc_{X,Y,\mc{T},p,s}\hab{[w]_{A_p}}.
\end{align*}
\end{theorem}

We will first prove a result assuming the $\ell^s$-boundedness of $\widetilde{\mc{T}}$ for a fixed $s\in [1, \infty)$.
\begin{proposition}
\label{proposition:mainprel}
Fix $1\leq s \leq r <p < \infty$ and let $X$ and $Y$ be $s$-convex Banach function spaces such that $X^s$ has the $\UMD$ property. Let $\mc{T}$ be a family of operators $\R^d\times \R^d \to \mc{L}(X,Y)$ such that
\begin{enumerate}[(i)]
\item $(x,y)\mapsto T(x,y)\xi$ is measurable for all $T \in \mc{T}$ and $\xi\in X$.
\item The family of operators $\widetilde{\mc{T}} := \cbrace{T(x,y): T\in \mc{T}, \, x,y\in \R^d}$ is $\ell^s$-bounded.
\end{enumerate}
Then $\mc{I}_{\mc{T}}$ is $\ell^s$-bounded from $L^p(w;X)$ to $L^p(w;Y)$ for all  $w \in A_{{p/s}}$ with
\begin{equation*}
  [\mc{I}_{\mc{T}}]_{\ell^s} \leq \inc_{X,p,r}\hab{[w]_{A_{{p/s}}}} \bracb{\widetilde{\mc{T}}}_{\ell^s}.
\end{equation*}
\end{proposition}

\begin{proof}
Let $(S_X,\mu_X)$ and $(S_Y,\mu_Y)$ be the measure spaces associated to $X$ and $Y$ respectively. For $j=1,\cdots,n$ take $I_{j}\in \mc{I}_{\mc{T}}$ and let $k_j \in \mc{K}$ and $T_j\in\mc{T}$ be such that $I_j$ = $I_{k_j,T_j}$. Fix simple functions $f_1,\cdots,f_n \in L^p(w;X)$ and note that
\begin{equation}\label{eq:sconvexI}
  \nrms{\has{\sum_{j=1}^n\abs*{I_{j} f_j}^s}^{1/s}}_{L^p(w;Y)} = \nrms{\sum_{j=1}^n\abs*{I_{j} f_j}^s}_{L^{p/s}\ha*{w;Y^s}}^{1/s}.
\end{equation}
Fix $x \in \R^d$, then by Hahn-Banach we can find a nonnegative $u_x \in (Y^s)^*$ with $\nrm{u_x}_{(X^s)^*} = 1$ such that
\begin{equation}
\label{eq:addweight}
  \nrms{\sum_{j=1}^n\abs*{I_j f_j(x)}^s}_{Y^s}=\sum_{j=1}^n\int_{S_Y}\abs*{I_j f_j(x)}^s u_x \dd \mu_Y.
\end{equation}
With Proposition \ref{proposition:Rubiols} we can then find a nonnegative $v_x \in (X^s)^*$ with $\nrm{v_x}_{(X^s)^*} \leq 1$ such that
\begin{equation}\label{eq:rubiocorapp}
  \int_{S_Y} \abs{T_j(x,y)\xi}^s v_x \dd\mu_Y \leq \bracb{\widetilde{\mc{T}}}_{\ell^s} \,\int_{S_X} \abs{\xi}^sv_x \dd\mu_X
\end{equation}
for $j=1,\cdots,n$, $y \in \R^d$ and $\xi \in X$. Since $\nrm{k_j}_{L^1(\R^d)}\leq 1$ by \cite[Lemma 4.3]{NVW15}, Holder's inequality yields
\begin{equation}\label{eq:jensen}
  \abs{I_j f_j(x)}^s\leq \int_{\R^d}  |k_j(x-y)|\abs{T_{j}(x,y)f_j(y)}^s\dd y.
\end{equation}
Applying \eqref{eq:jensen} and \eqref{eq:rubiocorapp} successively we get
\begin{align*}
\sum_{j=1}^n \int_{S_Y}\abs*{I_jf_j(x)}^s u_x \dd \mu_Y
&\leq \sum_{j=1}^n \int_{S_Y} \int_{\R^d} |k_j(x-y)| \abs{T_{j}(x,y)f_j(y)}^s\dd y \,  u_x \dd \mu_Y\\
&=\sum_{j=1}^n\int_{\R^d} \abs{k_j(x-y)} \int_{S_Y}\abs{T_{j}(x,y)f_j(y)}^s \, u_x \dd \mu_Y\dd y\\
&\leq \bracb{\widetilde{\mc{T}}}_{\ell^s}\sum_{j=1}^n\int_{S_X}\int_{\R^d} |k_j(x-y)| \abs{f_j(y)}^s \dd y \, v_x \dd \mu_X\\
&\leq \bracb{\widetilde{\mc{T}}}_{\ell^s} \, \nrms{\sum_{j=1}^n (\abs{k_j}*\abs{f_j}^s)(x)}_{X^s},
\end{align*}
using duality and $\nrm{v_x}_{(X^s)^*} \leq 1$ in the last step. We can now use the $\ell^1$-bounded\-ness result of Proposition \ref{proposition:kernel1sbdd}, since $(X^s)^*$ has the $\UMD$ property by \cite[Proposition 4.2.17]{HNVW17}. Combined with \eqref{eq:sconvexI} and
\eqref{eq:addweight} we obtain
\begin{align*}
 \nrms{ \has{\sum_{j=1}^n\abs*{I_j f_j}^s}^{{1/s}}}_{L^p(w;Y)}  & \leq \bracb{\widetilde{\mc{T}}}_{\ell^s} \, \nrms{\sum_{j=1}^n\abs{k_j}*\abs{f_j}^s}_{L^{{p/s}}\ha{w;X^s}}^{\frac{1}{s}}
 \\ &\leq \inc_{X,p/s} \hab{[w]_{A_{p/s}}} \,\bracb{\widetilde{\mc{T}}}_{\ell^s} \,
  \nrms{\sum_{j=1}^n \abs{f_j}^s}_{L^{{p/s}}\ha{w;X^s}}^{1/s}\\
 &\leq \inc_{X,p,r}\hab{[w]_{A_{p/s}}} \bracb{\widetilde{\mc{T}}}_{\ell^s}\nrms{\has{\sum_{j=1}^n \abs{f_j}^s}^{\frac{1}{s}}}_{L^{p}\ha{w;X}},
\end{align*}
where we can pick the increasing function $\phi$ in the last step independent of $s$, since the increasing function in Proposition \ref{proposition:kernel1sbdd} depends continuously on $p$. This can for example be seen by writing out the exact dependence on $p$ in Theorem \ref{proposition:HL} using \cite[Theorem 1.3]{HL18b} and \cite[Theorem 3.1]{Mo12}.
\end{proof}

Using this preparatory proposition, we will now prove Theorem \ref{theorem:main}.

\begin{proof}[Proof of Theorem \ref{theorem:main}]
Let $w \in A_p$. We shall prove the theorem in three steps.

\textbf{Step 1.}
First we shall prove the theorem very small $s>1$. By Proposition \ref{proposition:rubioUMD} we know that there exists a $\sigma_{X,Y} \in (1,p)$ such that $X$ and $Y$ are $s$-convex and $X^s$ has the $\UMD$ property for all $s \in [1,\sigma_X]$. By Lemma \ref{lemma:weights}\ref{it:weights3} we can then find a $\sigma_{p,w} \in (1, \sigma_{X,Y}]$ such that for all $s\in [1, \sigma_{p,w}]$
\begin{equation*}
  [w]_{A_{{p/s}}}\leq [w]_{A_{{p/\sigma_{p,w}}}}\leq \inc_{p}\hab{ [w]_{A_p}}
\end{equation*}
Let $\sigma_1 = \min\cbrace{\sigma_{X,Y},\sigma_{p,w}}$, then by Proposition \ref{proposition:mainprel} we know that $\mc{I}_{\mc{T}}$ is $\ell^s$-bounded from $L^p(w;X)$ to $L^p(w;Y)$  for $s\in (1, \sigma_1]$ with
\begin{equation}\label{eq:smalls}
\brac{\mc{I}_{\mc{T}}}_{\ell^s} \leq  \inc_{X,p,\sigma_{X,Y}}([w]_{A_{p/s}}) \bracb{\widetilde{\mc{T}}}_{\ell^s}  \leq  \inc_{X,Y,p}([w]_{A_{p}}) \bracb{\widetilde{\mc{T}}}_{\ell^s}.
\end{equation}

\textbf{Step 2.} Now we use a duality argument to prove the theorem for large $s<\infty$.
 As noted in the proof of Proposition \ref{proposition:kernel1sbdd}, we have $L^p(w;X)^* = L^{p'}(w';X^*)$ with $w' = w^{1-p'}$ under the duality pairing as in \eqref{eq:duality} and similarly for $Y$. Furthermore $X^*$ and $Y^*$ have the $\UMD$ property.

 It is routine to check that under this duality $I_{k,T}^* = I_{\tilde{k},\tilde{T}}$ with $\tilde{k}(x)=k(-x)$ and $\tilde{T}(x,y)=T^*(y,x)$ for any $I_{k,T} \in \mc{I}_{\mc{T}}$.
Trivially  $\tilde{k}\in \mc{K}$ if and only if $k \in \mc{K}$ and by Proposition \ref{proposition:kernel1sbdd}\ref{it:dualizels} the adjoint family $\widetilde{T}^*$ is $\ell^{\sigma'}$-bounded with
$$\bracb{\widetilde{\mc{T}}^*}_{\ell^{\sigma'}} = \bracb{\widetilde{\mc{T}}}_{\ell^\sigma}$$
for all $\sigma \in (1,\infty)$.
Therefore, it follows from step 1 that there is a $\sigma_2>1$
such that $\mc{I}_{\mc{T}}^*$ is $\ell^{s}$-bounded from $L^{p'}(w';Y^*)$ to $L^{p'}(w';X^*)$  for all $s\in (1, \sigma_2]$. Using  Proposition \ref{proposition:kernel1sbdd}\ref{it:dualizels} again, we deduce that $\mc{I}_{\mc{T}}$  is $\ell^s$-bounded from $L^{p}(w;X)$ to $L^{p}(w;Y)$ for all $s \in [\sigma_2', \infty)$ with
\begin{equation}\label{eq:larges}
\brac{\mc{I}_{\mc{T}}}_{\ell^s} = \brac{\mc{I}_{\mc{T}}^*}_{\ell^{s'}} \leq  \inc_{X,Y,p}\hab{[w]_{A_{p}}} \bracb{\widetilde{\mc{T}}}_{\ell^s}.
\end{equation}

\textbf{Step 3.} We can finish the prove by an interpolation argument for $s \in (\sigma_1,\sigma_2')$. By Proposition \ref{proposition:Rubiols}\ref{it:interpolatels} we get for $s \in (\sigma_1,\sigma_2')$ that $\mc{I}_\mc{T}$ is $\ell^s$-bounded from $L^{p}(w;X)$ to $L^{p}(w;Y)$ with
\begin{equation} \label{eq:inters}
  \brac{\mc{I}_{\mc{T}}}_{\ell^s} \leq  \inc_{X,Y,p}([w]_{A_{p}}) \, \max\cbraces{\bracb{\widetilde{\mc{T}}}_{\ell^{\sigma_1}}, \bracb{\widetilde{\mc{T}}}_{\ell^{\sigma_2'}}}.
\end{equation}
Now note that by Lemma \ref{lemma:weights} there is a $\sigma \in (1,\infty)$ such that $\sigma < \sigma_1,\sigma_2$ and $\sigma < s <\sigma'$ and
\begin{equation*}
  \sigma = 1 + \frac{1}{\inc_{p,s}([w]_{A_p})}.
\end{equation*}
Thus combining \eqref{eq:smalls}, \eqref{eq:larges} and \eqref{eq:inters} we obtain
\begin{align*}
  \brac{\mc{I}_{\mc{T}}}_{\ell^s} &\leq\inc_{X,Y,p}\hab{[w]_{A_p}} \, \max\cbraceb{ \bracb{\widetilde{\mc{T}}}_{\ell^\sigma},\bracb{\widetilde{\mc{T}}}_{\ell^{\sigma'}} }
  \leq\inc_{X,Y,\mc{T},p,s}\hab{[w]_{A_p}},
\end{align*}
using the fact that $t \mapsto \max\cbraceb{ \bracb{\widetilde{\mc{T}}}_{\ell^t},\bracb{\widetilde{\mc{T}}}_{\ell^{t'}} }$ is increasing for $t  \to 1$ by Proposition \ref{proposition:Rubiols}\ref{it:interpolatels}. This proves the theorem.
\end{proof}
\bigskip
\begin{remark}~
\begin{itemize}
 \item From Theorem \ref{theorem:main} one can also conclude that $\mc{I}_{\mc{T}}$ is $\mc{R}$-bounded, since $\mc{R}$- and $\ell^2$-boundedness coincide if $X$ and $Y$ have the $\UMD$ property, see e.g. \cite[Theorem 8.1.3]{HNVW17}.
  \item The $\UMD$ assumptions in Theorem \ref{theorem:main} are necessary. Indeed already if $X=Y$, $w=1$ and if $\widetilde{\mc{T}}$ only contains the identity operator, it is shown in \cite{KLW19} that the $\ell^2$-boundedness of $\mc{I}_{\mc{T}}$ implies the $\UMD$ property of $X$.
  \item The main result of \cite{GLV16} is Theorem \ref{theorem:main} for the special case $X=Y=L^q(S)$. In applications to systems of PDEs one needs Theorem \ref{theorem:main} on $L^q(S;\C^n)$ with $s=2$, see e.g. \cite{GV17b}. This could be deduced from the proof of \cite[Theorem 3.10]{GLV16}, by replacing absolute values by norms in $\C^n$. In our more general statement the case $L^q(S;\C^n)$ is included, since $L^q(S;\C^n)$ is a $\UMD$ Banach function space over $S\times \cbrace{1,\cdots,n}$

\end{itemize}
\end{remark}

If $X=Y$ is a rearrangement invariant Banach function space on $\R^e$, we can check the $\ell^\sigma$-bounded\-ness of $\widetilde{\mc{T}}$ for all $\sigma \in (1,\infty)$ by weighted extrapolation. Examples of such Banach function spaces are Lebesgue, Lorentz and Orlicz spaces. See \cite[Section 2.a]{LT79} for an introduction to rearrangement invariant Banach function spaces.

\begin{corollary}\label{corollary:main}
  Let $X$ be a rearrangement invariant $\UMD$ Banach function space on $\R^e$ and let $p,s \in (1,\infty)$. Let $\mc{T}$ be a family of operators $\R^d\times \R^d \to \mc{L}(X)$ such that
\begin{enumerate}[(i)]
\item $(x,y)\mapsto T(x,y)\xi$ is measurable for all $T \in \mc{T}$ and $\xi\in X$.
\item For some $q \in (1,\infty)$ and all $v \in A_q$ we have
$$\sup_{T\in \mc{T},\,x,y \in \R^d} \nrm{T(x,y)}_{\mc{L}(L^q(v))} \leq \inc_{\mc{T},q}\hab{[v]_{A_q}}$$
\end{enumerate}
Then $\mc{I}_{\mc{T}}$ is $\ell^s$-bounded on $L^p(w;X)$ for all $w \in A_p$  with
\begin{align*}
  \brac{\mc{I}_{\mc{T}}}_{\ell^s} &\leq\inc_{X,Y,\mc{T},p,q,s}\hab{[w]_{A_p}}.
\end{align*}
\end{corollary}

Note that in Corollary \ref{corollary:main} we need that $T(x,y)$ is well-defined on $L^q(v)$ for all $T \in \mc{T}$ and $x,y \in \R^d$. This is indeed the case, since $X \cap L^q(v)$ is dense in $L^q(v)$.

\begin{proof}
  Let $Y$ be the linear span of
  \begin{equation*}
    \cbrace{\ind_K \xi: K\subseteq \R^e \text{ compact},\, \xi \in X \cap L^\infty(\R^e)}.
  \end{equation*}
  Then $Y \subseteq L^q(v)$ for all $v \in A_p$ and $Y$ is dense in $X$ by order continuity. Define
  \begin{equation*}
    \mc{F}:= \cbraceb{\hab{\abs{T(x,y)\xi},\abs{\xi}}: T \in \mc{T},\, x,y\in \R^d, \, \xi \in Y}.
  \end{equation*}
  Note that $X$ has upper Boyd index $q_X <\infty$ by the $\UMD$ property (see \cite[Proposition 7.4.12]{HNVW17} and \cite[Section 2.a]{LT79}). So we can use the extrapolation result for Banach function spaces in \cite[Theorem 2.1]{CGM06} to conclude that for $\sigma \in (1,\infty)$
\begin{equation*}
  \nrms{\has{\sum_{j=1}^n \abs{T_j(x_j,y_j)\xi_j}^{\sigma}}^{1/\sigma}}_X \leq C_{\mc{T},q} \, \nrms{\has{\sum_{j=1}^n \abs{\xi_j}^{\sigma}}^{1/\sigma}}_X
\end{equation*}
for any $T_j \in \mc{T}$, $x_j,y_j \in \R^d$ and $\xi_j \in Y$ for $j=1,\cdots,n$. By the density this extends to $\xi_j \in X$, so $$\cbrace{T(x,y): x,y\in \R^d, T \in \mc{T}}$$ is $\ell^\sigma$-bounded for all $\sigma \in (1,\infty)$. Therefore the corollary follows from Theorem \ref{theorem:main}.
\end{proof}

\begin{remark}~\label{remark:maincorollary}
\begin{itemize}
\item A sufficient condition for the weighted boundedness assumption in Corollary \ref{corollary:main} is that $T(x,y)\xi \leq C \, M\xi$ for all $T\in \mc{T}$, $x,y \in \R^d$ and $\xi \in L^{q}(\R^e)$, which follows directly from \cite[Theorem 9.1.9]{Gr09}.
   \item Corollary \ref{corollary:main} holds more generally for $\UMD$ Banach function spaces $X$ such that the Hardy-Littlewood maximal operator is bounded on both $X$ and $X^*$ (see \cite[Theorem 4.6]{CMP11}). For example the variable Lebesgue spaces $L^{p(\cdot)}$ satisfy this assumption if $p_+,p_- \in (1,\infty)$ and $p(\cdot)$ satisfies a certain continuity condition, see \cite{CFN03,Ne04}.
 \item  The conclusion of Corollary \ref{corollary:main} also holds for $X(v)$ for all $v \in A_{p_X}$ where $p_X$ is the lower Boyd index of $X$ and $X(v)$ is a weighted version of $X$, see \cite[Theorem 2.1]{CGM06}.
\end{itemize}

\end{remark}

\bibliographystyle{plain}

\end{document}